\documentclass{amsart}
\usepackage{graphics}
\usepackage{graphicx}
\usepackage{subcaption}

\usepackage{epsfig,float,amsmath, amsfonts,amssymb,amsopn,amscd,amsthm}
\usepackage{mathrsfs}
 \newtheorem{theorem}{Theorem}[section]
 
 {\rm}
 \newtheorem{prop}[theorem]{Proposition}

 \newtheorem{defn}[theorem]{Definition}{\rm}
 
 \newtheorem{rem}[theorem]{Remark}
 \newtheorem{ex}{Example}
\numberwithin{equation}{section}

\def\x{\mathbf{x}}
\def\P{\mathbf{P}}

\def\R{\mathbb{R}}

\def\N{\mathbb{N}}

\def\K{\mathbf{K}}
\def\Q{\mathbf{Q}}
\def\M{\mathbf{M}}

\def\P{\mathbf{P}}

\def\A{\mathbf{A}}

\def\B{\mathbf{B}}

\def\X{\mathbf{X}}

\def\z{\mathbf{z}}

\def\y{\mathbf{y}}

\def\y{\mathbf{y}}

\def\om{\mathbf{\Omega}}

\title{Lebesgue and Gaussian measure of unions of basic semi-algebraic sets}
\author{Jean B. Lasserre}
\thanks{Jean B Lasserre: 7 Avenue du Colonel Roche, BP 54200, 31031 Toulouse cedex 4, France. \\
Tel: +66561336415; Fax: +33561336936; Email: {\tt lasserre@laas.fr}}

\address{LAAS-CNRS and Institute of Mathematics\\
University of Toulouse\\
LAAS, 7 avenue du Colonel Roche\\
31077 Toulouse C\'edex 4, France}
\email{lasserre@laas.fr}
\author{Youssouf Emin}
\thanks{Youssouf Emin: Ecole Polytechnique, 91 128 Palaiseau, Cedex, France\\
Email: {\tt youssouf.emin@polytechnique.edu}}
\address{Ecole Polytechnique\\
91 128 Palaiseau Cedex, France}
\email{youssouf.emin@polytechnique.edu}

\begin{document}
\maketitle

\section*{Abstract} 

Given a finite Borel measure $\mu$ on $\R^n$ and basic semi-algebraic sets $\om_i\subset\R^n$, $i=1,\ldots,p$, we provide a systematic numerical scheme 
to approximate as closely as desired $\mu(\bigcup_i\om_i)$, when all moments of $\mu$ are available (and finite). More precisely, we provide a hierarchy
of semidefinite programs whose associated sequence of optimal values is monotone 
and converges to the desired value from above. The same methodology applied to
the complement $\R^n\setminus (\bigcup_i\om_i)$ provides a monotone sequence that converges to the desired value from below.
When $\mu$ is the Lebesgue measure we assume that $\om:=\bigcup_i\om_i$ is compact and contained in a known box $\B:=[-a,a]^n$ and in this case
the complement is taken to be $\B\setminus\om$.
In fact, not only $\mu(\om)$ but also every finite vector of moments of $\mu_\om$ (the restriction of $\mu$ on $\om$) 
can be approximated as closely as desired, and so permits to approximate the integral on $\om$ of any given polynomial. \\
{\bf Keywords:} Lebesgue and Gaussian measure; semi-algebraic sets; moment problem and sums of squares; semidefinite programming; convex optimization\\
{\bf MSC:} 44A60 28A75 90C05 90C22

\section{Introduction}

Given a set $\om\subset\R^n$ and a finite Borel measure $\mu$ on $\R^n$, computing $\mu(\om)$ is a very challenging problem.
In fact even approximating the Lebesgue volume of a convex body $\om\subset\R^n$ (e.g. a polytope) is difficult; see e.g. 
Bollob\'as \cite{Bollobas} and Dyer and Frieze \cite{dyer1}.
However, in the latter case some efficient (non deterministic) algorithms with probabilistic guarantees are available and for more details the interested reader is referred to
e.g. Dyer et al. \cite{dyer2}, Cousins and Vempala \cite{cousins1,cousins2} and the references therein.  

In the non convex case no such algorithm is  available and one is left with approximating 
$\mu(\om)$ with Monte Carlo (or Quasi-Monte-Carlo) type methods as described in e.g. Niederreiter \cite{niederreiter}. That is, one first generates a sample of $N$ points in $\B$ following the distribution 
$\mu$ on $\B$ and then one {\it counts} the number $N_\K$ of points that fall into $\om$. This realization of the random variable $N_\K/N$ provides an estimate of $\mu(\om)$ but by no means an upper bound or a lower bound on $\mu(\om)$. Of course this method is quite fast, especially is small dimension.

Yet, as $\mu(\om)$  is indeed very difficult to compute exactly, a less ambitious but still useful goal would be to provide {\it upper} and/or {\it lower bounds} on $\mu(\om)$.
Even better, a  converging sequence of upper (or lower) bounds would be highly desirable.
This is the strategy proposed in Henrion et al. \cite{sirev} when $\om$ is a compact basic semi-algebraic set and $\mu$ is the Lebesgue measure.
In \cite{sirev} the authors have provided a (deterministic) numerical scheme which yields
a monotone sequence of upper bounds converging to $\mu(\om)$. It consists of solving a hierarchy of semidefinite programs of increasing size. 
By repeating the procedure but now with the complement $\B\setminus\om$, one also obtains a monotone sequence of 
lower bounds converging to $\mu(\om)$. However, even on typical $2$ or $3$-dimensional examples, the convergence was rather slow and the authors proposed a slight modification which turned out to be much more efficient; the convergence was much faster but unfortunately not monotone anymore.

\subsection*{Contribution}
The purpose of this paper is to introduce a deterministic method to approximate (in principle as closely as desired)
the measure $\mu(\om)$ of the union  $\om=\bigcup_i\om_i$ of finitely many basic semi-algebraic set.
The finite Borel measure $\mu$ is any  measure whose all moments are finite, e.g., the Lebesgue measure 
when $\om$ is compact, the Gausssian measure $d\mu=\exp(-\Vert\x\Vert^2)d\x$ for non-compact set $\om$.

The method is similar in spirit to the one in \cite{sirev} for a compact basic semi-algebraic set and the one in \cite{lass-gaussian}
for computing Gaussian measures of basic closed semi-algebraic sets (not necessarily compact), but with two important novelties.

$\bullet$ In contrast to \cite{sirev} and \cite{lass-gaussian}, we consider a finite {\it union} $\om$ of (non disjoint) basic semi-algebraic sets, which complicates matters significantly.

$\bullet$ We include a technique to accelerate the convergence different from the one described in \cite{sirev}. Indeed
in contrast to \cite{sirev}, it has the highly desirable feature to maintain the {\it monotone} convergence to $\mu(\om)$ which is essential if one wishes to obtain upper and lower bounds. It consists of
using moments constraints coming from a particular application of Stokes' theorem.

 In fact this numerical scheme  allows to approximate  not only $\mu(\om)$ but also
any fixed finite sequence of moments of the measure $\mu_\om$ (where $\mu_\om$
is the restriction of $\mu$ to $\om$). 
\begin{rem}
{\rm One might invoke the inclusion-exclusion principle which states that
\begin{equation}
\label{toto}
\mu (\bigcup_{i=1}^p \om_i) = \sum_{j=1}^p (-1)^{j+1} \sum_{1 \le i_1 < ... < i_j \le p} \mu (\om_{i_1}\cap ... \cap \om_{i_j}),
\end{equation}
so that in principle it suffices to compute (or approximate) $\mu (\om_{i_1}\cap ... \cap \om_{i_j})$ for all possible 
intersections of the $\om_j$'s, e.g.  by the approach of \cite{sirev} or \cite{lass-gaussian}. But this approach has two major drawbacks.
First there are possibly $2^p$ such sets  and secondly,
to compute an upper bound one has to compute an upper bound 
for  such intersections with an odd number of elementary sets $\om_{i_j}$, and a lower bound 
for  such intersections with an even number of elementary sets. The latter lower bound in turn is obtained 
by computing an upper bound for the complement. This makes the whole procedure tedious and complicated.
Finally, Bonferroni's inequalities also provide a (finite) sequence of upper and lower bounds on $\mu(\om)$
but computing those bounds involves sums similar to the right-hand-side of (\ref{toto}),
hence with the same drawbacks just mentioned. Our proposed technique is direct with no
partial computation on intersections of elementary sets $\om_{i_j}$.
}\end{rem}

Of course, the technique described in this paper is computationally expensive. In particular, its applicability  is limited by the 
performance of the state-of-the-art semidefinite solvers because the size of the semidefinite programs
increases fast with the rank in the hierarchy. Therefore it makes 
its application limited to small dimensional problems ($n\leq 3,4$). For higher dimensions
only a few steps in the hierarchy can be implemented and therefore only upper and lower bounds (possibly crude) are expected. 
But the reader should keep in mind that the problem is very difficult and to the best of our knowledge 
we are not aware of an algorithm (at least at this level of generality) which provides certified upper and lower bounds
with such convergence properties (even for convex sets and in particular for non compact sets $\om$). In our opinion 
this methodology should be viewed as complementary to (rather than competing with) probabilistic methods.

\section{Notation, definitions and preliminary results}
\subsection{Notation and definitions}
Let $\R[\x]$ be the ring of polynomials in the variables $\x=(x_1,\ldots,x_n)$. Denote by $\R[\x]_d\subset\R[\x]$ the vector space of polynomials of degree at most $d$, which has dimension $s(d):=\binom{n+d}{d}$, with e.g., the usual canonical basis $(\x^\gamma)_{\gamma\in\N^n_d}$ of monomials, where $\N^n_d := \{\gamma\in\N^n \,:\, \vert\gamma\vert\leq d\}$, $\N$ is the set of natural numbers including $0$ and $\vert\gamma\vert := \sum_{i=1}^n\gamma_i$.
Also, denote by $\Sigma[\x]\subset\R[\x]$ (resp. $\Sigma[\x]_d\subset\R[\x]_{2d}$) the cone of sums of squares (s.o.s.) polynomials (resp. s.o.s. polynomials of degree at most $2d$). 
If $f\in\R[\x]_d$, we write $f(\x)=\sum_{\gamma\in\N^n_d}f_\gamma \x^\gamma$ in the canonical basis and denote by $\boldsymbol{f}=(f_\gamma)_\gamma\in\R^{s(d)}$ its vector of coefficients. 
Finally, let $S^n$ denote the space of $n\times n$ real symmetric matrices, with inner product $\langle \A,\B\rangle ={\rm trace}\,\A\B$. We use the notation $\A\succeq 0$ (resp. $\A\succ 0$) to denote that $\A$ is positive semidefinite (definite). With $g_0:=1$, the quadratic module $Q(g_1,\ldots,g_m)\subset\R[\x]$ generated by polynomials $g_1,\ldots,g_m$, is defined by:
\[
Q(g_1,\ldots,g_m)\,:=\,\left\{\sum_{j=0}^m\sigma_j\,g_j\::\:\sigma_j\in\Sigma[\x]\,\right\}.
\]
\begin{defn}[Archimedean assumption]
\label{definition-1}
The quadratic module $Q(g_1,\ldots,g_m)$ is {\it Archimedean} if there exists $M>0$ such that the quadratic polynomial
$\x\mapsto g_{m+1}:=M-\Vert\x\Vert^2$ belongs to $Q(g_1,\ldots,g_m)$. Notice that
$g_{m+1}\in Q(g_1,\ldots,g_m)$ is an {\it algebraic certificate} that the set $\K:=\{\x: g_j(\x)\geq0,\:j=1,\ldots,m\}$ is compact. 
\end{defn}
If the set  $\K:\{\x:g_j(\x)\geq0,\:j=1,\ldots,m\}$ is compact  then $\Vert\x\Vert^2\leq M$ for some $M>0$,
and one may always include the redundant quadratic constraint $\theta(\x):=M-\Vert\x\Vert^2\geq0$ in the definition of $\K$ without changing $\K$.
Then the quadratic module $Q(g_1,\ldots,g_{m+1})$ is Archimidean.

\subsection*{Moment and localizing matrix}
With a real sequence $\y=(y_{\gamma})_{\gamma\in\N^n_d}$, one may associate the (Riesz) linear functional $L_{\y}:\R[\x]_d\to\R$ defined by
\[f\,\left(=\sum_\gamma f_\gamma\,\x^\gamma\right)\quad \mapsto L_{\y}(f)\,:=\,\sum_\gamma f_\gamma\,y_\gamma,\]
Denote by $\M_d(\y)$ the {\it moment} matrix associated with $\y$, the real symmetric matrix with rows and columns indexed in the basis of monomials $(\x^\gamma)_{\gamma\in \N^n_{d}}$, and with entries:
\[
\M_d(\y)(\alpha,\beta)\,:=\,L_{\y}(\x^{\alpha+\beta})\,=\,y_{\alpha+\beta},\qquad \forall\,\alpha,\beta\in\N^n_{d}.
\]
Next, given $g\in\R[\x]$, denote by $\M_d(g\,\y)$ the {\it localizing} moment matrix associated with $\y$ and $g$, the real symmetric matrix with rows and columns indexed in the basis of monomials $(\x^\gamma)_{\gamma\in \N^n_{d}}$, and with entries:
\[
\M_d(g\,\y)(\alpha,\beta)\,:=\,L_{\y}(g(\x)\,\x^{\alpha+\beta})\,=\,\sum_{\gamma}g_\gamma\,y_{\alpha+\beta+\gamma},\qquad \forall\,\alpha,\beta\in\N^n_{d}.
\]
If $\y=(y_{\gamma})_{\gamma\in\N^n}$ is the sequence of moments of some Borel measure $\mu$ on $\R^n$ then $\M_d(\y)\succeq0$ for all $d\in\N$. 
However the converse is not true in general and it is related to the well-known fact that there are positive polynomials that are not sums of squares.
Similarly, if the support of $\mu$ is contained in $\{\x:g(\x)\geq0\}$ then $\M_d(g\,\y)\succeq0$ for all $d\in\N$.
For more details the interested reader is referred to e.g. \cite[Chapter 3]{lass-book-icp}.

Given a Borel set $\om\subset\R^n$ let $\mathcal{M}(\om)$ be the space of finite {\it signed} Borel measures on $\om$ and let
$\mathcal{M}(\om)_+\subset\mathcal{M}(\om)$ be the convex cone of finite Borel measures on $\om$.

\subsection{The measure of a basic semi-algebraic set}
\label{sirev}
Let $\B,\K\subset\R^n$ with $\B\supset\K$ and let $\mu$ be a finite Borel measure 
whose support is $\B$.  (Typically $\mu$ is the Lebesgue measure on a box $\B$
and one wishes to compute the Lebesgue volume ${\rm vol}(\K)$; alternatively $\B=\R^n$, 
$\mu$ is the Gaussian measure $d\mu=\exp(-\Vert\x\Vert^2)d\x$
and one wishes to compute $\mu(\K)$.)

\subsection*{An infinite-dimensional linear program $\P$} 

Let $f\in\mathbb{R}[\x]$ be positive almost everywhere on $\K$ and consider the following infinite-dimensional \textbf{LP} problem :
\begin{center}
\begin{align}
\P : \quad f^*=\underset{\phi}{\mbox{sup  }} \Big \{\int_\K f\, d\phi : \lambda \le \mu; \mbox{ } \phi \in \mathcal{M}(\K)_+ \Big \} \label{LPINF_P}
\end{align}
\end{center}
\begin{theorem}[\cite{sirev}]
\label{LPINF_THM}
The measure $\phi^*=\mu_\K$ (the restriction of $\mu$ to $\K$) is the unique optimal solution of $\P$. In particular, if $f(\x)=1$ for all $\x$, then 
$f^*=\mu(\K)$.
\end{theorem}

\subsection*{Semidefinite relaxations}
Of course problem $\P$ in (\ref{LPINF_THM}) is infinite-dimensional and cannot be solved directly. However,
when $\K$ is a basic semi-algebraic set then Theorem \ref{LPINF_THM} can be further exploited. 
So given $(g_j)_{j=1}^m\subset\R[\x]$, let $\K\subset\R^n$ be the basic semi-algebraic set
\begin{equation}
\label{setk}
\K\,=\,\{\,\x\in\R^n:\:g_j(\x)\,\geq\,0,\quad j=1,\ldots,m\,\},
\end{equation}
assumed to nonempty and compact. Let $\B\supset\K$ and let $\mu$ be a finite Borel measure whose all moments 
$\z=(z_\alpha)$ with 
\[z_\alpha\,=\,\int_\B\x^\alpha\,d\mu(\x),\qquad\alpha\in\N^n,\]
are available in closed form or can be computed.

To approximate $f^*$ as closely as desired in \cite{sirev} the authors propose to solve the following hierarchy $(\Q_d)_{d\in\N}$ of 
semidefinite programs\footnote{A semidefinite program (SDP) is a conic convex optimization problem
with a remarkable modeling power. It can be solved efficiently 
(in time polynomial in its input size) up to arbitrary precision fixed in advance; see e.g. Anjos and Lasserre \cite{handbook}} indexed by $d\in\N$:

\begin{equation}
\label{primal-sdp}
\Q_d:\begin{array}{rl}
\rho_d=\displaystyle\sup_\y & \{\,L_{\y}(f): \\
\mbox{s.t.}& \M_d(\y)\,\succeq0; \:\M_d(\z-\y) \succeq 0\\
& \M_{d-r_{j}}(g_{j}\,\y) \succeq 0, j = 1,\ldots,m\}.
\end{array}\end{equation}
Observe that $\Q_d$ is a relaxation of $P$, and so $\rho_d\geq\mu(\K)$ for all $d$. In addition,
the sequence $(\rho_d)_{d\in\N}$ is monotone non increasing.
The dual of (\ref{primal-sdp})
is the semidefinite program:
\begin{equation}
\label{dual-sdp}
\Q^*_d:\quad \rho^*_d=\displaystyle\inf_{p\in\R[\x]_{2d}} \, \{\,\int_\B p\,d\mu:\: p-f\geq0\mbox{ on $\K$};\quad p\in\Sigma[\x]_d\,\},
\end{equation}
and by weak duality, $\rho_d\leq\rho^*_d\leq f^*$ for all $d$.

\begin{theorem}[\cite{sirev}]
Assume that $Q(g_1,\ldots,g_m)$ is Archimedean. Then $\rho_d\to f^*$ as $d\to\infty$. If $\K$ has nonempty interior then 
$\rho^*_d=\rho_d$ and (\ref{dual-sdp}) has an optimal solution $p^*\in\R[\x]_{2d}$.
\end{theorem}
So when $f=1$, $(\rho_d)_{d\in\N}$ provides us with a monotone sequence of upper bounds on $f^*=\mu(\K)$. Unfortunately
the convergence is rather slow as observed on several numerical examples. This is because
in the dual (\ref{dual-sdp}) the optimal solution $p^*\in\R[\x]_{2d}$ tries to approximate from above (in $L_1(\B,\mu)$)
the discontinuous function $1_\K$, which implies an annoying Gibb's phenomenon\footnote{The Gibbs' phenomenon appears at a jump discontinuity when one approximates a piecewise $C^1$ function with a continuous function, e.g. by its Fourier series.}. To remedy this problem
the authors in \cite{sirev} propose to use a polynomial $f$, nonnegative on $\K$ and which vanishes on $\partial \K$.
In this case the convergence $\rho_d\to \int_\K f\,d\mu$ as $d\to\infty$ is still monotone and if
$\y^d=(y^d_\alpha)_{\alpha\in\N^n_{2d}}$ denotes an optimal solution of (\ref{primal-sdp}) then $y^d_0\to\mu(\K)$ as $d\to\infty$.
However, while faster than with $f=1$, the latter convergence of $y^d_0$ to $\mu(\K)$ is not monotone anymore, a rather annoying feature
which prevents from obtaining a non increasing sequence of upper bounds.

\section{Main result}

\subsection*{The context} 
Let $\B\subset\R^n$ be a box, and for every $i=1,...,p$, let $\om_i := \{\,\x\in\R^n :g_{ij}(x) \ge 0, j=1,\ldots,m_i\}$,
for some polynomials $(g_{ij})\subset\R[\x]$.
Assume that $\B$ has been chosen so as to satisfy:
\begin{align}
\om := \bigcup_{i=1}^p \om_i \subset \B.  \label{SEMIALGDEF}
\end{align} 
The goal is to provide a numerical scheme to approximate as closely as desired the Lebesgue volume
$\mu(\K)$. (We will see how to adapt the methodology to also approximate
as closely as desired $\mu(\K)$ when $\K$ is not necessarily compact and $\mu$ is a Gaussian measure.)
One possible approach described below is to use the powerful inclusion-exclusion principle and/or
the associated Bonferroni inequalities.

\subsection{The inclusion-exclusion principle and Bonferroni Inequalities}
Let :
\[S_k \,:=\, \sum_{1 \le i_1 < ... < i_k \le p} \mu (\om_{i_1}\cap ... \cap \om_{i_k}),\quad k=1,\ldots,p.\]
By the inclusion-exclusion principle,

\begin{align*}
\mu (\bigcup_{k=1}^p \om_k) = \sum_{k=1}^p (-1)^{k+1}S_k,
\end{align*}
which allows us to work with intersections of the $\om_k$'s only. In addition, the Bonferroni inequalities state that 

\begin{align*}
\mu (\bigcup_{i=1}^p \om_i) &\le \sum_{j=1}^{2k+1} (-1)^{j+1}S_j &\forall 2k+1 \le p\\
& \ge \sum_{j=1}^{2k} (-1)^{j+1}S_j &\forall 2k \le p
\end{align*}
which provides sequences of (increasingly tighter) upper and lower bounds.\\

Therefore to compute $\mu(\om)$ we only have to compute the measure of the intersection 
$\Theta_{i_1,\ldots,i_k}:=\displaystyle\bigcap_{j=1,\ldots,k}\om_{i_j}$, for all $1 \le i_1 < ... < i_k \le p$. 
Notice that there are $2^p$ such sets.
As each
$\Theta_{i_1,\ldots,i_k}\subset\B$ is a compact basic semi-algebraic set, one may apply the methodology described in \S\ref{sirev},
to obtain a sequence $(\rho^{(i_1,\ldots,i_k)}_d)_{d\in\N}$ which converges to $\mu(\Theta_{i_1,\ldots,i_k})$ as $d\to\infty$, and therefore
\[\lim_{d\to\infty}\left(\sum_{i=1}^p(-1)^{k+1}\sum_{1\leq i_1<\ldots<i_k\leq p}\rho_d^{i_1,\ldots,i_k}\right)\,=\,\mu(\om).\]
Notice that the convergence is not monotone non increasing even if one solves (\ref{primal-sdp}) with $f=1$ because 
we sum up negative and positive terms. To maintain the monotone convergence (when $f=1$)
it suffices to compute a lower bound on the complement $\B\setminus\Theta_{i_1,\ldots,i_k}$ when $k$ is even. However
as already mentioned the convergence is expected to be rather slow. 

To accelerate the convergence one may use $f=\prod_{j=1}^k\prod_{\ell=1}^{m_{i_j}}g_{i_j\ell}$ when one solves (\ref{primal-sdp})
with $\om=\Theta_{i_1,\ldots,i_k}$ as $f\geq0$ on $\Theta_{i_1,\ldots,i_k}$ and 
$f=0$ on $\partial\Theta_{i_1,\ldots,i_k}$. 
But then the convergence 
\[\lim_{d\to\infty}\left(\sum_{i=1}^p(-1)^{k+1}\sum_{1\leq i_1<\ldots<i_k\leq p}y_{d,0}^{i_1,\ldots,i_k}\right)\,=\,\mu(\om),\]
(where $\y^{i_1,\ldots,i_k}_d=y^{i_1,\ldots,i_k}_{d,\alpha}$ is an optimal solution of (\ref{primal-sdp}) with $\om=\Theta_{i_1,\ldots,i_k}$) is not monotone anymore.

\subsection{A direct approach}

In this section we describe a direct approach with two distinguishing features:
\begin{itemize}
\item It does not use the inclusion-exclusion principle and the need to approximate
$\mu(\bigcap_{j=1}^k\om_{i_j})$ for all $2^p$ such sets.
\item The convergence to $\mu(\om)$ (and also to $\mu(\B\setminus\om))$ is monotone non increasing, that is, we
can compute two sequences $(\overline{\omega}_d)_{d\in\N}$ and $(\underline{\omega}_d)_{d\in\N}$ such that:
\[\underline{\omega}_d\,\leq\,\mu(\om)\,\leq\,\overline{\omega}_d,\quad d\in\N;\quad
\mu(\om)\,=\,\lim_{d\to\infty}\underline{\omega}_d\,=\,\lim_{d\to\infty}\overline{\omega}_d.\]
\end{itemize} 

Recall that any finite number of moments
\begin{center} 
\begin{align*}
\mu_{\alpha} = \int_{\mathbb{B}} \textbf{x}^\alpha d\mu(\textbf{x}),\quad\alpha\in\N^n,
\end{align*}  
\end{center}
are either available in closed-form or can be obtained numerically.

\subsection*{A multi infinite-dimensional linear program $\Q$}
As in \S \ref{sirev} we first introduce an infinite-dimensional \textbf{LP} problem $\Q$ whose unique optimal solution is the restriction of $\mu$ on $\om$ (and whose dual has a clear interpretation). 

Let $f\in\mathbb{R}[\x]$ be positive almost everywhere on $\om$ and consider the following infinite-dimensional \textbf{LP} problem :

\begin{equation}
\label{LPINF_Q}
\Q : \quad f^*=\sup_{\phi_1,\ldots,\phi_p}
\Big \{\sum_{i=1}^p\int_{\om_i} f d\phi_i : \sum_{i=1}^p \phi_i \le \mu; \mbox{ } \phi_i \in \mathcal{M}(\om_i)_+,\:i=1,\ldots,p\Big \}. 
\end{equation}

\begin{theorem} \label{OPTSOLQ}
Problem $\Q$ has an optimal solution $(\phi^*_1,\ldots,\phi^*_p)$ and 
every optimal solution satisfies $\sum_{i=1}^p \phi^*_i = \mu_\om$, where
$\mu_\om$ is the restriction of $\mu$ to $\om$.
\end{theorem}

\begin{proof}
We first prove that $\Q$ has an optimal solution. The set $\Delta_\mu:=\{\phi\in\mathcal{M}(\R^n)_+:\phi\leq\,\mu\}$
is weakly sequentially compact; see e.g. Dunford \& Schwartz \cite[Theorem 1, p. 305]{dunford}. Therefore let
$(\phi^k_1,\ldots,\phi^k_p)_{k\in\N}$ be a maximizing sequence of feasible solutions of $\Q$. There exists
a subsequence $(k_\ell)_{\ell\in\N}$ such that for every $i=1,\ldots,p$, $\phi^{k_\ell}_i\stackrel{w}{\to}\phi^*_i$ for some
$\phi^*_i\in\mathcal{M}(\R^n)_+$. The above weak convergence and $\displaystyle\int_{\om_i^c}\,d\phi_i^{k_\ell}=0$ implies
$\displaystyle\int_{\om_i^c}\,d\phi_i^*=0$, that is, $\phi^*_i\in\mathcal{M}(\om_i)_+$ for all $i=1,\ldots,p$.
Weak convergence again implies $\sum_{i=1}\phi^*_i\leq\mu$ and so $(\phi^*_1,\ldots,\phi^*_p)$ is a feasible solution of $\Q$.
Finally weak convergence also implies 
\[f^*\,=\,\lim_{\ell\to\infty}\sum_{i=1}^p\int f\,d\phi^{k_\ell}_i\,=\,\sum_{i=1}^p\lim_{\ell\to\infty}\int f\,d\phi^{k_\ell}_i\,=\,
\sum_{i=1}^p\int f\,d\phi^*_i,\]
which proves that $(\phi^*_1,\ldots,\phi^*_p)$ is an optimal solution of $\Q$.

We next prove that $f^*=\int  fd\mu_\om$.
Indeed, firstly observe that for every feasible solution $(\phi_1,\ldots,\phi_p)$ of $\Q$,
$\sum_{i=1}^p\int f d\phi_i \le \int_\om fd\mu= \int fd\mu_\om$. On the other hand,
for every $i=1,\ldots,p$, denote by $\theta_i$ the measurable function defined on $\om$ by :
\begin{center}
\begin{align*}
\x\mapsto \theta_i(\x) = \frac{1}{\vert\{j \in \{1,\ldots,p\} : \x \in \om_j \}\vert} 1_{\om_i}(\x),\qquad\x\in\om.
\end{align*} 
\end{center}
The (discontinuous) functions $(\theta_i)_{i=1,\ldots,p}$ form a {\it partition of unity} subordinate to the open cover $\bigcup_i{\rm int}(\om_i)$.
For every $i=1,\ldots,p$, let $\phi^*_i\in\mathcal{M}(\om_i)_+$ be the finite Borel measure defined by:
\begin{equation}
\label{phistar}
\phi^*_i(C) \,:=\,\int_C \theta_i(\x)\,d\mu(\x),\qquad\forall C\in\mathcal{B}(\R^n).
\end{equation}
Hence, $\sum_{i=1}^p\phi^*_i(C)=\int_C\sum_{i=1}^p \theta_i(\x)d\mu(\x) = \int_C 1_{\om}(\x)(\x)d\mu(\x) = \mu_\om(C) \le \mu(C)$. Therefore
$(\phi^*_1,\ldots,\phi^*_p)$ is a feasible solution of $\Q$ such that $\sum_{i=1}^p\phi^*_i=\mu_\om$, and so
$\sum_{i=1}^p\int fd\phi^*_i=f^*$, i.e., $(\phi^*_1,\ldots,\phi^*_p)$ is an optimal solution of $\Q$.
In fact, {\it every} optimal solution $(\phi^*_1,\ldots,\phi^*_p)$ of $\Q$ satisfies 
$\sum_{i=1}^p \int fd\phi^*_i=f^*$, and therefore $\phi^*:=\sum_{i=1}^p\phi^*_i\in\mathcal{M}(\om)_+$
is an optimal solution of $\sup_\phi\{\int fd\phi:\phi\leq\mu;\:\phi\in\mathcal{M}(\om)_+\}$. By Theorem \ref{LPINF_THM}
this solution $\phi^*$ is unique, which yields the desired result.
\end{proof}

\subsection*{A hierarchy of semidefinite relaxations}

Let $\z = (z_\alpha)_{\alpha\in\N^n}$ be the sequence of all moments of $\mu$
, that is,
\begin{equation}
\label{moments-mu}
z_\alpha\,:=\,\int \x^\alpha\,d\mu(\x),\quad \alpha\in\N^n.
\end{equation} 
Let $\B\subset\R^n$ be a box and $\om\subset\B$ be a compact semi-algebraic as in (\ref{SEMIALGDEF}).
With no loss of generality and possibly after scaling, we may and will assume that $\B\subset [-1,1]^n$
and $\mu$ is a probability measure. Therefore $\vert z_\alpha\vert\leq 1$ for all $\alpha\in\N^n$.

Let $r_{ij}= \lceil {\rm deg}(g_{ij})/2 \rceil$ and let $f\in \mathbb{R}[\x]$ be a given polynomial 
positive almost everywhere on $\om$ (and define $r_{00} := \lceil {\rm deg}(f)/2 \rceil$). For $d \ge d_0:=\mbox{max}_{i,j} r_{ij}$, consider the following 
hierarchy of semidefinite programs $(\Q_d)$ indexed by $d\in\N$ :
\begin{equation}
\label{primal-newsdp}
\Q_d:\quad\begin{array}{rl}
\rho^f_d=\displaystyle\sup_{\y^1,\ldots,\y^d}&\Big \{\, \displaystyle\sum_{i=1}^p L_{\y^i}(f)\\
\mbox{s.t.}&\M_d(\z - \sum_{i=1}^p \y^i) \succeq 0;\:\M_d(\y^i)\succeq0,\quad i=1,\ldots,p\\
& \M_{d-r_{ij}}(g_{ij}\,\y^i) \succeq 0, \quad j=1,\ldots,m_i;\:i=1,\ldots,p\,\Big \}.
\end{array}
\end{equation}

Observe that $\rho^f_d\geq f^*$ for all $d\in\N$. Indeed, if $(\z^1,\ldots,\z^p)$
is the sequence of moments of an optimal solution $(\phi^*_1,\ldots,\phi^*_p)$ of  $\Q$ in (\ref{LPINF_Q})
then $(\z^1,\ldots,\z^p)$ is also a feasible solution of $\Q_d$.

\begin{theorem} \label{thm:SDP}
Consider the semidefinite programs $(\Q_d)$, $d\geq d_0$. Then :

\noindent
(i) $\Q_d$ has an optimal solution and the associated sequence of optimal values $(\rho^f_d)_{d\in\N}$ is monotone non increasing
and converges to $f^*$, that is:
\[\rho^f_d\,\downarrow\, f^*\,=\,\int_\om f\,d\mu,\quad \mbox{as $d\rightarrow \infty$.}\]
\noindent
(ii) Let $(\y^{1,d},\ldots,\y^{p,d})$ be an optimal solution of $\Q_d$. Then for each $\alpha\in\N^n$:
\begin{equation}
\label{convergence}
\lim_{d\to\infty}\displaystyle \sum_{i=1}^py^{i,d}_\alpha \,=\, z^*_\alpha\,=\,\int_\om \x^\alpha\,d\mu.\end{equation}
and in particular $\displaystyle\lim_{d\to\infty} \sum_{i=1}^py^{i,d}_0=\mu(\om)$.
\end{theorem}

\begin{proof}
For a sequence $\y= (y_\alpha)$, let $\tau_d(\y) = \max_{i=1,\ldots,n}L_{\y}(x_i^{2d})$ and recall
that if $\M_d(\y)\succeq0$ then $\vert y_\alpha\vert\leq \max[y_0,\max_iL_\y(x_i^{2d})]$ for every $\alpha\in\N^n_{2d}$; see \cite[Proposition 3.6]{lass-book-icp}.
Next, observe that from $\M_d(\z - 
\sum_{i=1}^p\y^{i,d}) \succeq 0$ and $\M_d(\y^i) \succeq 0$, 
\[\M_d(\z - \y^{i,d}) \,\succeq\, \M_d(\sum_{j \neq i} \y^j) \succeq 0,\quad i=1,\ldots,n.\]
Hence the diagonal elements $z_{2\alpha} - y^{i,d}_{2\alpha}$ are all nonnegative which in turn implies
$\tau_d(\y^{i,d}) \le \tau_d(\z)\leq 1$ for all $i=1,\ldots,n$. 
As $z_0=1$ then  by \cite[Proposition 3.6 ]{lass-book-icp} $\vert y^{i,d}_{\alpha}\vert \le 1$ 
for every $\alpha \in \mathbb{N}^n_{2d}$, and so the feasible set of semidefinite program
$\Q_d$ is closed, bounded, hence compact, and therefore $\Q_d$ has an optimal solution.

Next, let $(\y^{1,d},\ldots,\y^{p,d})$ be an optimal solution of $\Q_d$ and by completing with zeros, make $(\y^{1,d},\ldots,\y^{p,d})$ 
an element of the unit ball of  $(\ell_\infty)^p$ (where
$\ell_\infty$ is the Banach space of bounded sequences, equipped with the sup-norm). 
As $(\ell_\infty)^p$ is the topological dual of $(\ell_1)^p$, by the Banach-Alaoglu Theorem, there exists $(\y^{1,*},..,\y^{p,*}) \in (\ell_\infty)^p$ and a subsequence $\{d_k\}$ such that $(\y^{1,d_k},\ldots,\y^{p,d_k}) \rightarrow (\y^{1,*},\ldots,\y^{p,*})$ as 
$k\rightarrow \infty$, for the weak $\star$ topology $\sigma((\ell_\infty)^p, (\ell_1)^p)$. In particular,
\begin{equation}
\label{CVS}
\lim_{k \to\infty}\, y_\alpha^{i,d_k}\, =\, y_\alpha^{i,*},\quad \forall \alpha \in \mathbb{N}^n, \forall i \in \{1,..,p\}.
\end{equation}
Next let $d\in \mathbb{N}$ be fixed arbitrary. From the pointwise convergence (\ref{CVS}) we also obtain $\M_d(\y^{i,*}) \succeq 0$ and $\M_d(\z-
\sum_{i=1}^p\y^{i,*}) \succeq 0$ for every $i=1,\ldots,p$. Similary, $\M_{d-r_{ij}}(g_{ij}\y^{i,*}) \succeq 0$ for every $i$ and $j$. 
As $d$ was arbitrary, by Putinar's Positivistellensatz \cite{putinar}, $\y^{i,*}$ has a representing measure $\phi_i$ supported on 
$\om_i$ for all $i=1,\ldots,p$, and $\sum_{i=1}^p \phi_i \le \mu$. 
In particular from (\ref{CVS}), as $k \to \infty$,
\[f^* \le\rho^f_{d_k} \,=\, \sum_{i=1}^p L_{\y^{i,d_k}}(f) \downarrow \sum_{i=1}^p L_{\y^{i,*}}(f) = \sum_{i=1}^p \int f\,d\phi_i.\]
Therefore $(\phi_1,\ldots,\phi_p)$ is admissible for problem $\Q$ with value $\sum_{i=1}^p \int fd\phi_i \ge f^*$, and so
$(\phi_1,\ldots,\phi_p)$ is an optimal solution of $\Q$. 
Finally, by Theorem \ref{OPTSOLQ}, $\sum_{i=1}^p \phi_i = \mu_\om$. And so for each $\alpha\in\N^n$: 
\[\lim_{k\rightarrow \infty}\,\sum_{i=1}^py^{i,d_k}_\alpha\, = \,\sum_{i=1}^py^{i,*}_\alpha\, =\,  z^*_\alpha\,=\,\int_\om\x^\alpha\,d\mu(\x).\]
As the converging subsequence $(d_k)_{k\in\N}$ was arbitrary, it follows that in fact the whole sequence $(\sum_{i=1}^py^{i,d}_\alpha)_d$ 
converges to $z_\alpha$, for all $\alpha\in\N^n$, that is, (\ref{convergence}) holds.
\end{proof}

\subsection*{The dual of $\Q_d$}
Let $g_{i0}(\x)=1$ for all $\x\in\R^n$, $i=1,\ldots,p$. The dual of the semidefinite program $\Q_d$ is the semidefinite program:
\begin{equation}
\label{dual-newsdp}
\Q^*_d:\quad\begin{array}{rl}
(\rho^f _d)^*=\displaystyle\inf_{q,\sigma_{ij}}&\Big\{\displaystyle\int_\B q\,d\mu:\\
\mbox{s.t.}&q-f\,=\,\displaystyle\sum_{j=0}^{m_i}\sigma_{ij}\,g_{ij},\quad i=1,\ldots,p\\
&\sigma_{ij}\in\Sigma[\x]_{d-r_{ij}},\quad j=0,\ldots,m_i;\:i=1,\ldots,p\\
&q\in\Sigma[\x]_d\Big\}.
\end{array}
\end{equation}

\begin{prop}
\label{slater}
Assume that for every $i=1,\ldots,p$, both $\om_i$ and $\B\setminus\om_i$ have nonempty interior.
Then there is no duality gap between (\ref{primal-newsdp}) and its dual (\ref{dual-newsdp}), that is,
$\rho^f _d=(\rho^f _d)^*$ 
for all $d\geq d_0$. Moreover (\ref{dual-newsdp}) has an optimal solution $(q^*,(\sigma_{ij}^*)$.
\end{prop}

\begin{proof}
Let $(\phi^*_1,\ldots,\phi^*_p)$ be the measures defined in (\ref{phistar}) the proof of the Theorem \ref{OPTSOLQ} and let 
$\y^i_d$ be the sequence of moments up to degree $d$ of $\phi^*_i$, $i=1,\ldots,p$.
As every  $\om_i$ has nonempty interior, then clearly $\M_d(\y^i) \succ 0$ and $\M_{d-r_{ij}}(g_{ij}\y^i) \succ 0$ for every 
$j= 1,\ldots,m_i$ and $i = 1,\ldots,p$. As $\B\setminus\om$ also has nonempty interior then $\M_d(\z - \sum_{i=1}^p\y^i) \succ 0$. 
Therefore Slater's condition holds for $\Q_d$. In addition, the set of admissible solution of $\Q_d^*$ is nonempty (set $q=f$ and $\sigma_{ij}=0$ for all $i,j$), and therefore  
a standard result in conic convex optimization yields the desired result\footnote{In fact as the set of optimal solutions of (\ref{primal-newsdp})
is compact, the absence of a duality gap between (\ref{primal-newsdp}) and
(\ref{dual-newsdp}) also follows from \cite{strong} without the conditions ${\rm int}(\om_i)\neq\emptyset$ and
${\rm int}(\B\setminus \om_i)\neq\emptyset$.}.

\end{proof}

As in the case of a basic closed semi-algebraic set, when $f$ is the constant function $1$ the convergence
$\rho^f_d\to f^* =\mu(\om)$ is monotone non increasing, a highly desirable feature. However 
in typical examples this convergence is rather slow. Again one may take for $f$
a function that is nonnegative on $\om$ and which vanishes on $\partial\om$. This accelerates the convergence 
both $\rho^f_d\to f^*$ and $\sum_i y^{id}_0\to \mu(\om)$ as $d\to\infty$, but if by construction the former is monotone
non increasing,  the latter is not monotone anymore, a rather annoying feature if the goal is to obtain a 
converging sequence of {\it upper bounds}. In the next section we describe a technique
that allows to accelerate significantly the convergence $\sum_i y^{id}_0\to \mu(\om)$ as $d\to\infty$, while maintaining
its monotone non increasing character.

\subsection{Convergence improvement using Stokes' formula}

In this section we show how to improve significantly the monotone non increasing
convergence of $\rho_d^1$ (i.e. $\rho^f_d$ with $f=1$) to $\mu(\om)$. To do this we will use Stokes' theorem for integration and in the sequel, to avoid technicalities we assume that $\om \subset \mathbb{R}^n$ is the closure of its interior, i.e., $\om=\overline{{\rm int}(\om)}$. 
The basic idea is simple to express in informal terms.

\begin{center}
{\it Since we know in advance that  $(\phi^*_1,\ldots,\phi^*_p)$ in  (\ref{phistar}) is an optimal solution of problem $\Q$,
every additional information in terms of linear constraints on the moments of $\phi^*_i$
can be included in $\Q$ without changing its optimal value. BUT when included
in the relaxation $\Q_d$ it will provide useful additional constraints that restrict the feasible set of $\Q_d$ and so
 make its optimal value necessarily smaller.}
\end{center}

Suppose for the moment that $\om$ is compact
with smooth boundary, and assume that the measure $\mu$ has a density $h$ with respect to Lebesgue measure $d\x$ of the form
$q(\x)\exp(r(\x))1_{\B}(\x)$ for some polynomial $r,q\in \R[\x]$. 
Let $X$ be a given vector field and $f\in \R[\x]$. Then Stokes' theorem states:
\[\int_{\om} {\rm Div}(X)\,f(\x)\, h(\x) d\x + \int_{\om} \langle X, \nabla (f(\x)\,h(\x))\rangle d\x = \int_{\partial K} 
\langle X,\vec{n}_\x \rangle f(\x)h(\x) d\sigma(\x),\]
where $\vec{n}_\x$ is outward pointing normal at $\x\in\partial\om$, and $\sigma$ is the 
$(n-1)$-dimensional Hausdorff measure on $\partial\om$. In particular if $f$ vanishes on $\partial\om$ 
and $\X=e_k\in\R^n$ (where $e_k(j)=\delta_{k=j}$) Stokes' formula becomes
\begin{equation}
\label{stokes}
\int_{\om} \frac{\partial}{\partial x_k}\,(f(\x)\,h(\x))\, d\x = 0.
\end{equation}
To exploit (\ref{stokes}) in our particular context where $\om$ is defined in (\ref{setk}), 
let $g = \prod_{i=1}^p \prod_{j=1}^{m_i} g_{ij}$ and let $\x\mapsto f(\x) = \x^\alpha\,g(\x)q(\x)$
with $\alpha\in\N^n$ arbitrary. Then $f$ vanishes on $\partial\om$ and 
on $\partial \om_{i_1,\ldots,i_s}:=\om_{i_1}\cap\cdots\cap\om_{i_s}$ for all $1\leq i_1<\ldots<i_s\leq p$, $s=1,\ldots,p$.
Hence by (\ref{stokes}):
\begin{eqnarray}
\label{stokes-alpha-1}
\int_{\om} p_{\alpha,k} (\x)\,\underbrace{q(\x)\,\exp(r(\x))\,d\x}_{d\mu(\x)}\,=\,0\\
\label{stokes-alpha-2}
\int_{\om_{i_1,\ldots,i_s}} p_{\alpha,k} (\x)\,\underbrace{q(\x)\,\exp(r(\x))\,d\x}_{d\mu(\x)}&=&0,
\end{eqnarray}
for all $1\leq i_1<\ldots<i_s\leq p$, $s=1,\ldots,p$, where 
\[p_{\alpha,k} (\x) = q(\x)\frac{\partial}{\partial_{x_k}} (\x^\alpha g(\x)) + 2\x^\alpha g(\x)\,\frac{\partial}{\partial_{x_k}}q(\x) + \x^\alpha g(\x)q(\x)\frac{\partial}{\partial_{x_k}} r(\x).\]
Recalling how $\phi^*_i$ is defined in (\ref{phistar}), it can be written as 
\[\phi^*_i\,=\,\sum_{s=1}^p\sum_{1\leq i_1<\cdots<i_s\leq p}\phi^*_{i, i_1,\ldots,i_s},\]
where each $\phi^*_{i, i_1,\ldots,i_s}$ is supported on $\om_i\cap\om_{i_1,\ldots,i_s}$ and has a constant density w.r.t. $\mu$.
Therefore, for every $i=1,\ldots,p$:
\begin{equation}
\label{stokes-alpha-3}
\int_{\om_i}p_{\alpha,k} (\x)\,d\phi^*_i\,=\,0,\quad\forall \alpha\in\N^n;\:k=1,\ldots,n.
\end{equation} 
Hence (\ref{stokes-alpha-3}) provides additional useful information
on the optimal solution $(\phi^*_1,\ldots,\phi^*_p)$ of $\Q$
defined in (\ref{phistar}). Namely it translates into 
\[L_{\y^i}(p_{\alpha,k})\,=\,0,\quad\forall \alpha\in\N^n;\:k=1,\ldots,n;\: i=1,\ldots,p,\]
i.e., linear constraints on the moments of $\phi^*_i$, for every $i=1,\ldots,p$.

 Plugging this additional linear constraints
on the moments of $\phi^*_i$ into the relaxation $\Q_d$, yields the following new hierarchy 
of SDP-relaxation $(\Q_d^{{\rm stokes}})_{d\geq d_0}$:
\begin{equation}
\label{primal-newsdp-stokes}
\begin{array}{rl}
\rho^{{\rm Stokes}}_d=\displaystyle\sup_{\y^1,\ldots,\y^d}&\Big \{\, \displaystyle\sum_{i=1}^p \y^i_0\\
\mbox{s.t.}&\M_d(\z - \sum_{i=1}^p \y^i) \succeq 0;\:\M_d(\y^i)\succeq0,\quad i=1,\ldots,p\\
& \M_{d-r_{ij}}(g_{ij}\,\y^i) \succeq 0, \quad j=1,\ldots,m_i;\:i=1,\ldots,p\\
&L_{\y^i}(p_{\alpha,k})\,=\,0,\quad k=1,\ldots,n;\:\vert\alpha\vert\leq 2d-{\rm deg}(p_{\alpha,k})\\
&i=1,\ldots,p\Big \}.
\end{array}
\end{equation}
By construction $\rho^1_d\geq\rho^{{\rm Stokes}}_d\geq\mu(\om)$ holds for every $d\geq d_0$,
and the analogue of Theorem \ref{thm:SDP} (with $f=1$) reads:

\begin{theorem} \label{thm:SDP-Stokes}
Consider the semidefinite programs $(\Q_d^{{\rm Stokes}})$, $d\geq d_0$, defined in (\ref{primal-newsdp-stokes}). Then :

\noindent
(i) $\Q_d^{{\rm Stokes}}$ has an optimal solution and the associated sequence of optimal values $(\rho^{{\rm Stokes}}_d)_{d\in\N}$ is monotone non increasing
and converges to $\mu(\om)$, that is:
\[\rho^{{\rm Stokes}}_d\,\downarrow\, \mu(\om),\quad \mbox{as $d\rightarrow \infty$.}\]
\noindent
(ii) Let $(\y^{1,d},\ldots,\y^{p,d})$ be an optimal solution of $\Q^{{\rm Stokes}}_d$. Then for each $\alpha\in\N^n$:
\begin{equation}
\label{convergence-stokes}
\lim_{d\to\infty}\displaystyle \sum_{i=1}^py^{i,d}_\alpha \,=\, z^*_\alpha\,=\,\int_\om \x^\alpha\,d\mu.\end{equation}
\end{theorem}
The proof being almost a verbatim copy of that of Theorem \ref{thm:SDP}, is omitted.

The important feature of Theorem \ref{thm:SDP-Stokes} is that we now have 
the monotone non increasing convergence $\rho^{{\rm Stokes}}_d\downarrow\mu(\om)$
(compare with (\ref{convergence}) (with $\alpha=0$) in Theorem \ref{thm:SDP}).

\subsection{Gaussian measure of non compact sets $\om$}

So far {\rm Theorem \ref{thm:SDP} and Theorem \ref{thm:SDP-Stokes} have been given for
$\mu$ supported on a box $\B$, and so only for sets $\om$ in (\ref{SEMIALGDEF}) that are compact.

It turns out that for a Gaussian measure $\mu$ of (possibly non-compact) sets $\om=\bigcup_i\om_i$,
Theorem \ref{thm:SDP} (resp. Theorem \ref{thm:SDP-Stokes}) is still valid with exactly the same statement
and exactly the same semidefinite 
relaxations (\ref{primal-newsdp}) (resp. (\ref{primal-newsdp-stokes})), except that now
 $\z=(z_\alpha)$ is the vector of moments of the Gaussian measure $\mu$ (instead of
 the moments of the Lebesgue measure on $\B$ previously).

However in the gaussian case the proof of Theorem \ref{thm:SDP}(i)-(ii) and Theorem \ref{thm:SDP-Stokes}(i)-(ii)
uses quite different arguments (some already used in \cite{lass-gaussian} for a basic semi-algebraic set). Indeed
as $\om$ is not necessarily compact :

- The uniform bound $\sup_\alpha\vert\z_\alpha\vert\leq1$ is not valid any more for the relaxations $\Q_d$ and $\Q^{Stokes}_d$.

- One cannot invoke Putinar's Positivstellensatz \cite{putinar} any more.

- The standard version of Stokes' theorem where $\om$ is compact cannot be invoked anymore either.

The new arguments that we need are the following:

$\bullet$ A crucial fact is that $\mu$ satisfies Carleman's condition
\begin{equation}
\label{carleman}
\sum_{k=1}^\infty \left(\int_{\R^n} x_i^{2k}\,d\mu(\x)\right)^{-1/2k}\,=\,+\infty,\quad i=1,\ldots,n.\end{equation}
Then a sequence $\y=(y_\alpha)_{\alpha\in\N^n}$ such that $\M_d(\y)\succeq0$ for all $d\in\N$, and 
\[\sum_{k=1}^\infty L_\y(x_i^{2k})^{-1/2k}\,=\,+\infty,\quad i=1,\ldots,n,\]
has a unique representing measure $\phi$ on $\R^n$ which is moment determinate; see for instance \cite[Proposition 3.5, p. 60]{lass-book-icp}.

$\bullet$ If in addition $\M_d(h\,\y)\succeq0$ for all $d\in\N$ (where $h\in\R[\x]$), and as $\phi$
satisfies (\ref{carleman}), then $h(\x)\geq0$ for all $\x$ in the support of $\phi$; see Lasserre \cite{newlook}. This argument is used to show that $\phi$
is supported on $\om$.

$\bullet$ To obtain a version of Stokes for  non-compact set $\om$ with boundary $\partial\om$, we invoke  a limiting argument
that uses (the standard) Stokes's theorem on the compact $\om\cap \B(0,M)$ (where $\B(0,M)=\{\x:\Vert\x\Vert \leq M\}$).
Letting  $M\to\infty$ and using the Monotone and Bounded Convergence theorems yields the desired result.
For more details the reader is referred to \cite{lass-gaussian} where such arguments have been used in the case of a {\it basic} semi-algebraic set.

Finally it is worth emphasizing that this methodology also works for {\it any} measure $\mu$ that satisfies (\ref{carleman})
(and  whose moments are known or can be computed); an important spacial case is the exponential measure on the positive orthant $\R^n_+$.

\begin{rem}
{\rm As mentioned above, in \cite{lass-gaussian} the first author has already used Stokes' formula to accelerate the convergence
of a hierarchy of semidefinite relaxations to approximate the Gaussian measure $\mu(\om)$ of a {\it basic} semi-algebraic set
$\om$, not necessarily compact. 
The important and non trivial novelty here is that 
(i) $\om=\bigcup_{i=1}^p\om_i$ is now a {\it union} of basic semi-algebraic sets,
and (ii) even if this complicates matters significantly,
we are still able to work with measures $\phi_i$, each supported on $\om_i$ (a basic semi-algebraic set). 
It turns out that $\mu(\om)=\sum_{i=1}^p\phi_i^*$ where is 
each $\phi^*_i$ has a piecewise constant density w.r.t. $\mu$ (constant on each of the possible intersections
$\om_i\cap \om_{i_1,\ldots,i_p}$). 
By using a family of polynomials that all vanish on the boundary of {\it each} $\om_i\cap\om_{i_1,\ldots,i_p}$, we
can exploit Stokes' Theorem on each piece and sum up to obtain a family of linear constraints on the moments of $\phi^*_i$.
}\end{rem}

\section{Numerical experiments and discussion}

For illustration purposes we have applied the methodology on a few (simple) examples.
We report some numerical experiments carried out in Matlab and GloptiPoly3 \cite{gloptipoly}, a software package for manipulating and solving generalized problems of moments. The SDP problems were solved with SeDuMi 1.1R3. 

\subsection{Lebesgue volume of a union of two ellipsoids}

We first consider a simple example of two ellipsoids in $\R^2$ where the exact value
$\mu(\om)$ can be computed exactly so that we can compare with our upper bounds. So we want to compute the Lebesgue measure of 
$\om = \om_1 \cup \om_2$ with $\om_1 = \{(x_1,x_2) \in \mathbb{R}^2 : \frac{x_1^2}{4} + x_2^2 \le 1\}$ and $\om_2 = \{(x_1,x_2) \in \mathbb{R}^2 : 
\frac{x_2^2}{4} + x_1^2 \le 1\}$. In this example  we take $\B:=[-2,2]^2$.

\begin{figure}[htb]
\centering 
\includegraphics[width=0.65\columnwidth]{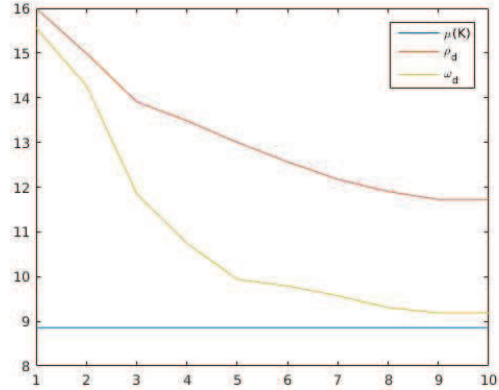} 
\caption{$n=2$: Lebesgue measure of a union of 2 ellipsoids
\label{fig:vol1} }
\end{figure}

The results are displayed in  the Figure~\ref{fig:vol1} with:  in \textbf{orange} the approximation of the Lebesgue volume
$\mu(\om)$ without using Stokes' formulas, in \textbf{red} the approximation when using Stokes' formulas and in \textbf{blue} the 
exact value of $\mu(\om)$.\\

We next consider a union of two ellipsoids in dimension $n=3$. Let $\om_1 = \{\x\in \mathbb{R}^3 : x_1^2 + 4x_2^2 + 4x_3^2\le 1\}$, $\om_2 = \{\x\in \mathbb{R}^3 : x_2^2 + 4x_1^2 + 4x_3^2\le 1\}$, $\om=\om_1\cup\om_2$ and $\mathbf{B} = [-1,1]^3$. 
Results are displayed in Figure~\ref{fig:vol2}.
In both examples one can check that the convergence is much faster when using  Stokes' formula.

\begin{figure}[htb]
\centering 
\includegraphics[width=0.65\columnwidth]{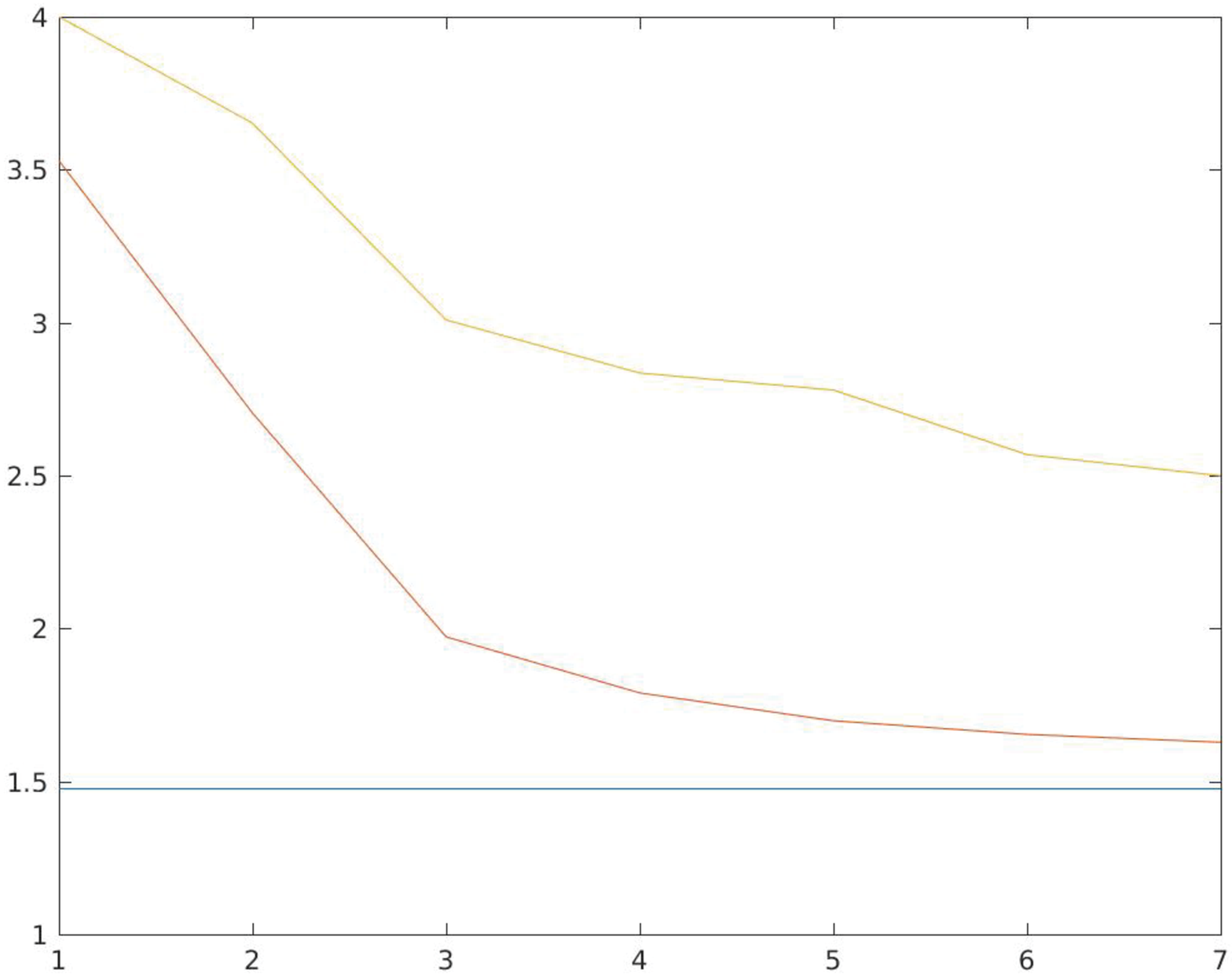} 
\caption{$n=3$: Lebesgue measure of a union of 2 ellipsoids \label{fig:vol2} }
\end{figure}

\subsection{Lebesgue measure a union of three ellipsoids}

We next  consider a union of three ellipsoid in dimension $n=2$, with: 
\[\om_1 = \{\x \in \mathbb{R}^2 : (x_1,x_2).\left[\begin{array}{cc}
\frac{16}{9} & 0\\
0 & 4
\end{array}\right]\left(\begin{array}{c}
x_1\\
x_2
\end{array}\right) \le 1\},\]
\[\om_2 = \{(x_1,x_2) \in \mathbb{R}^2 : \frac{1}{9}(x_1-0.1,x_2-0.1).\left[\begin{array}{cc}
31 & 5\sqrt{3}\\
5\sqrt{3} & 21
\end{array}\right]\left(\begin{array}{c}
x_1-0.1\\
x_2-0.1
\end{array}\right) \le 1\},\]
\[\om_3 = \{(x_1,x_2) \in \mathbb{R}^2 : \frac{1}{9}(x_1+0.1,x_2-0.1).\left[\begin{array}{cc}
31 & -5\sqrt{3}\\
-5\sqrt{3} & 21
\end{array}\right]\left(\begin{array}{c}
x_1+0.1\\
x_2-0.1
\end{array}\right) \le 1\},\]
and $\B = [-1,1]^2$. In Figure \ref{fig:vol3} we also compare our results with those obtained 
when using \textit{Bonferroni} inequalities. 
In \textbf{red} the upper bounds obtained by solving $\textbf{Q}^{Stokes}$, in \textbf{orange} the lower bounds
obtained by solving $\textbf{Q}^{Stokes}$ for the complement, and in \textbf{blue} the upper bounds obtained by using Bonferroni inequalities. (For a fair comparison, for each relaxation in Bonferroni case we also use appropriate Stokes' constraints.)
\begin{figure}[htb]
\centering 
\includegraphics[width=0.65\columnwidth]{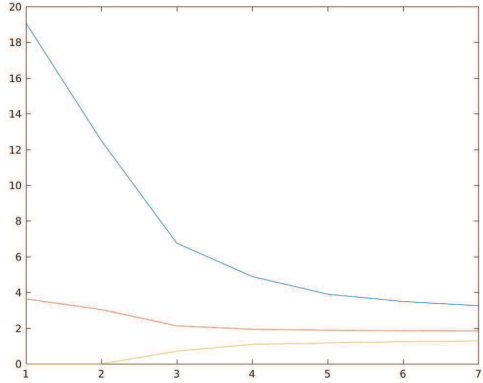} 
\caption{$n=2$: Lebesgue measure of a union of 3 ellipsoids \label{fig:vol3} }
\end{figure}

\subsection{Examples for the Gaussian measure} 

In this section we consider the Gaussian measure $d\mu = \exp(-\frac{\left\|\textbf{x}\right\|^2}{\sigma^2})d\textbf{x}$ 
with variance $\sigma^2 = 0.8$. For each example we have computed two upper-bounds and two lower-bounds for $\mu(\om)$. The first 
(resp. second) upper-bound $\overline{\rho}_d$ (resp. $\overline{\rho}^{Stokes}_d$) 
is obtained by  solving the semidefinite relaxation 
$\textbf{Q}_d$ (resp. $\textbf{Q}^{Stokes}_d$). Similary, the lower-bounds $\underline{\rho}_d$ (resp. $\underline{\rho}^{Stokes}_d$) are obtained from  upper bounds 
for the complement $\R^n\setminus\om$. The respective relative error-gap
are denoted by $\epsilon_d= \frac{\overline{\rho}_d - \underline{\rho}_d}{\overline{\rho}_d}$ and
$\epsilon^{Stokes}_d = \frac{\overline{\rho}^{Stokes}_d - \underline{\rho}^{Stokes}_d}{\overline{\rho}^{Stokes}_d}$.

\begin{ex}
\label{ex-gauss-1}
{\rm 
In this example $\om $ is the union of two ellipsoids.
Let $\om := \om_1 \cup \om_2$ with $\om_1 = \{\textbf{x} \in \mathbb{R}^2 : (\textbf{x} - \textbf{u})^T \textbf{A}_1 (\textbf{x} - \textbf{u}) \le 1\}$ and $\om_2 = \{\textbf{x} \in \mathbb{R}^2 : (\textbf{x} - \textbf{v})^T \textbf{A}_2 (\textbf{x} - \textbf{v}) \le 1\}$ for the values
\[
\begin{array}{ccccc}
\textbf{u} = &(0,0), &(0.1,0.5), &(0.5,0.5)
\end{array}
\]
$\textbf{v} = (1,0)$, \[\textbf{A}_1 = \left[\begin{array}{cc}
1 & 0\\
0 & \frac{1}{4}
\end{array}
\right]
\mbox{ and } \textbf{A}_2 = \left[\begin{array}{cc}
\frac{1}{4} & 0\\
0 & 1
\end{array}
\right]\]
In this case $\rho:=\mu(\om$ can be computed exactly and so 
we have displayed the values of the relative errors denoted by $\epsilon_d = \frac{\overline{\rho}_d - \underline{\rho}}{\overline{\rho}}$
and $\epsilon^{Stokes}_d = \frac{\overline{\rho}^{Stokes}_d - \underline{\rho}}{\overline{\rho}}$ respectively, depending on whether or not we have used Stokes' formula. As one can see in Table \ref{tbl:1} for  a reasonable value $d=10$
the relative error (when using Stokes' formula) is quite good. The respective behaviors are displayed in Figure \ref{gauss-1}.

\begin{table}[h] 
\bgroup
\def\arraystretch{1.5}
\begin{tabular}{|l|c|c|c|}
\hline
& $\textbf{u} = (0,0)$ & $\textbf{u} = (0.1,0.5)$ & $\textbf{u} = (0.5,0.5)$\\
\hline 
$\overline{\rho}_{10}$ & $1.9649$ & $1.9554$ & $1.9484$\\
\hline 
$\underline{\rho}_{10}$ & $1.6129$ & $1.5752$ & $1.5369$\\
\hline
$\epsilon_{10}$ & $18\%$ & $19\%$ & $21\%$ \\
\hline 
$\overline{\rho}^{Stokes}_{10}$ & $1.8571$ & $1.8308$ & $1.8156$\\
\hline 
$\underline{\rho}^{Stokes}_{10}$ & $1.7948$ & $1.7746$ & $1.7618$\\
\hline
$\epsilon^{Stokes}_{10}$ & $3\%$ & $3\%$ & $3\%$ \\
\hline 
\end{tabular}
\egroup
\caption{Example \ref{ex-gauss-1}: Values of $\overline{\rho}_{10}$, $\underline{\rho}_{10}$, $\overline{\rho}_{10}^{Stokes}$, $\underline{\rho}_{10}^{Stokes}$, $\epsilon_{10}$ and $\epsilon_{10}^{Stokes}$ \label{tbl:1}}
\end{table}
\begin{figure}[h]
\centering
\includegraphics[width=0.65\linewidth]{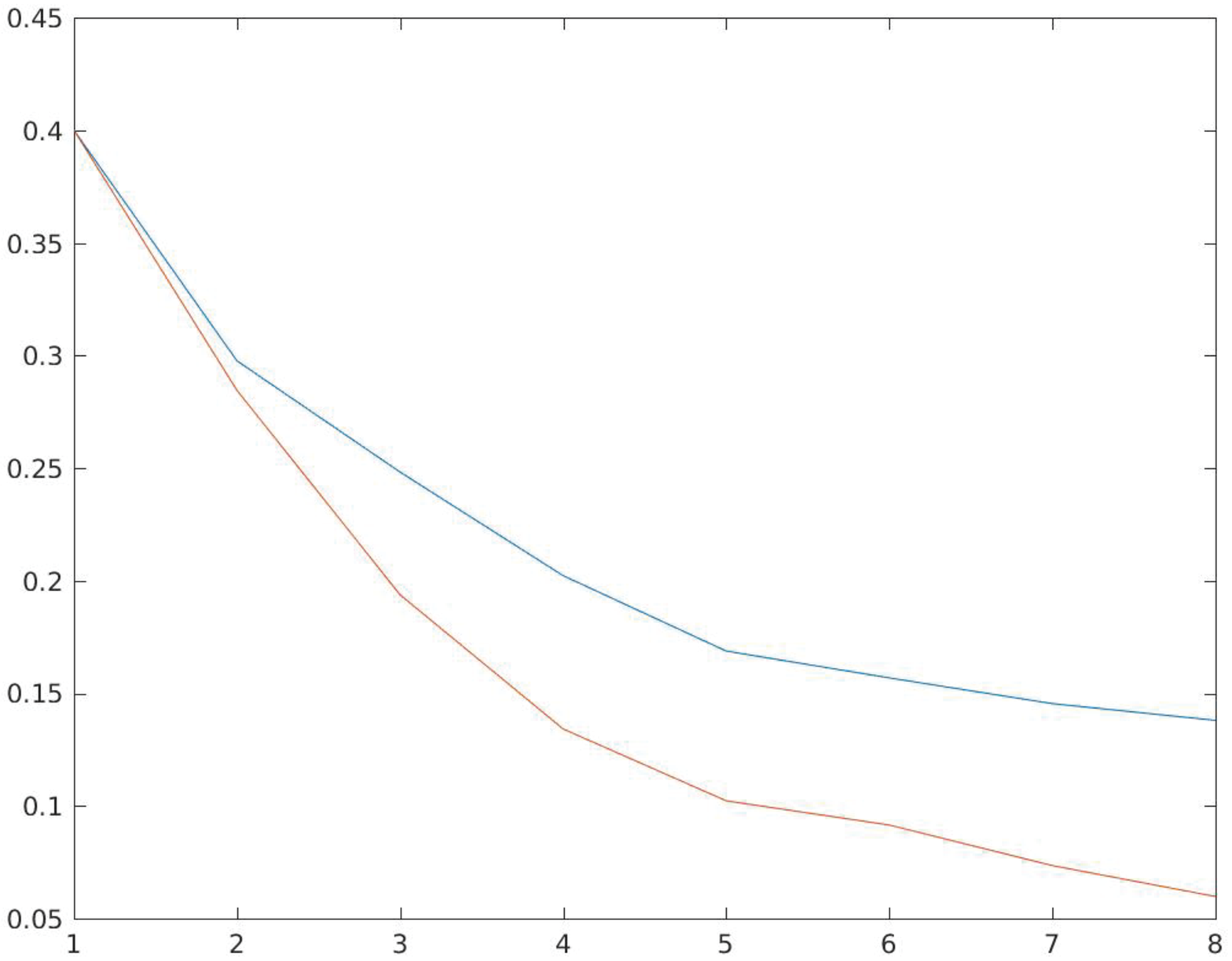}
\caption{Example \ref{ex-gauss-1}: Relative errors $\epsilon_d$ (\textbf{blue}) and $\epsilon_d^{Stokes}$ (\textbf{red})
\label{gauss-1}}
\end{figure}
}\end{ex}
\newpage

\begin{ex}
\label{ex-gauss-2}
{\rm 
Consider $\om = \om_1 \cup \om_2$ with $\om_1 = \{\textbf{x} \in \mathbb{R}^2 : (\textbf{x} - \textbf{u})^T \textbf{A}_1 (\textbf{x} - \textbf{u}) \le 1\}$ and $\om_2 = \{\textbf{x} \in \mathbb{R}^2 : (\textbf{x} - \textbf{v})^T \textbf{A}_2 (\textbf{x} - \textbf{v}) \le 1\}$ for the values 

$
\textbf{u} = (0,0),
$
$\textbf{v} = (-2,0)$, \[\textbf{A}_1 = \left[\begin{array}{cc}
\frac{1}{16} & 0\\
0 & 1
\end{array}
\right]
\mbox{ and } \textbf{A}_2 = \left[\begin{array}{cc}
\frac{1}{4} & \frac{1}{2}\\
\frac{1}{2} & -1
\end{array}
\right]\]
In this case $\om$ is not a compact set as it is unbounded. The results for $d=9$ displayed in Table \ref{tbl:ex2}
show that a good value is already obtained when using Stokes' formula. The respective behaviors of $\epsilon_d$ and $\epsilon^{Stokes}_d$
displayed in Figure \ref{fig:ex2} also show that using Stokes' formula yields a significant improvement.
\begin{table}[h] 
\bgroup
\def\arraystretch{1.5}
\begin{tabular}{|c|c|c|c|c|c|}
\hline 
$\overline{\rho}_{9}$ & $\underline{\rho}_{9}$ & $\epsilon_{9}$ & $\overline{\rho}^{Stokes}_{9}$ & $\underline{\rho}^{Stokes}_{9}$ & $\epsilon^{Stokes}_{9}$\\
\hline
$2.0038$ & $1.8252$ & $8.9\%$ & $1.9347$ & $1.9019$ & $1.7\%$\\ 
\hline
\end{tabular}
\egroup
\caption{Example \ref{ex-gauss-2}: Bounds and relative gap for $d=9$\label{tbl:ex2}}
\end{table}

\begin{figure}[ht]
\includegraphics[width=0.65\columnwidth]{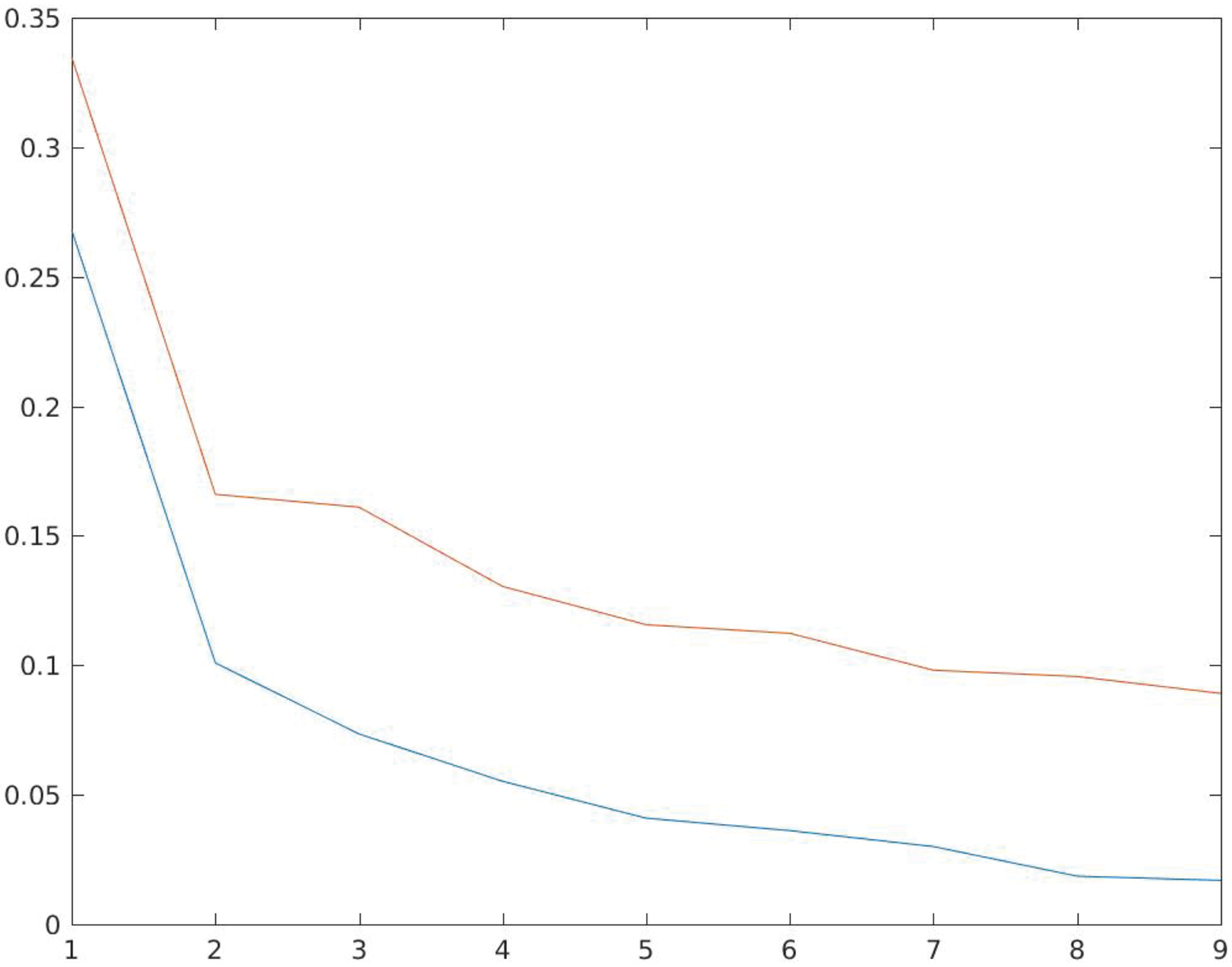} 
\caption{Example \ref{ex-gauss-2}: Relative error $\epsilon_d$ (\textbf{blue}) and $\epsilon^{Stokes}_d$ 
(\textbf{red}) \label{fig:ex2} }
\end{figure}
}\end{ex}

\begin{ex}
\label{ex-gauss-3}
{\rm Consider $\om = \om_1 \cup \om_2$ with $\om_1 = \{\textbf{x} \in \mathbb{R}^2 : (\textbf{x} - \textbf{u})^T \textbf{A}_1 (\textbf{x} - \textbf{u}) \le 1\}$ and $\om_2 = \{\textbf{x} \in \mathbb{R}^2 : (\textbf{x} - \textbf{v})^T \textbf{A}_2 (\textbf{x} - \textbf{v}) \le 1\}$ for the values 

$
\textbf{u} = (0,0),
$
$\textbf{v} = (-2,0)$, \[\textbf{A}_1 = \left[\begin{array}{cc}
-\frac{1}{16} & 0\\
0 & 1
\end{array}
\right]
\mbox{ and } \textbf{A}_2 = \left[\begin{array}{cc}
\frac{1}{4} & \frac{1}{2}\\
\frac{1}{2} & -1
\end{array}
\right]\]
Again $\om$ is not compact. The results in Table \ref{tbl:ex3} and the respective behaviors of $\epsilon_d$ and $\epsilon^{Stokes}_d$
displayed in Figure \ref{fig:ex3} confirm that using Stokes' formula yields a significant improvement.
\begin{table}[h] 
\bgroup
\def\arraystretch{1.5}
\begin{tabular}{|c|c|c|c|c|c|}
\hline 
$\overline{\rho}_{9}$ & $\underline{\rho}_{9}$ & $\epsilon_{9}$ & $\overline{\rho}^{Stokes}_{9}$ & $\underline{\rho}^{Stokes}_{9}$ & $\epsilon^{Stokes}_{9}$\\
\hline
$2.0046$ & $1.8342$ & $8\%$ & $1.9542$ & $1.9083$ & $2\%$\\ 
\hline
\end{tabular}
\egroup
\caption{Example \ref{ex-gauss-3}: Bounds and relative gap for $d=9$ \label{tbl:ex3}}
\end{table}

\begin{figure}[ht]
\includegraphics[width=0.65\columnwidth]{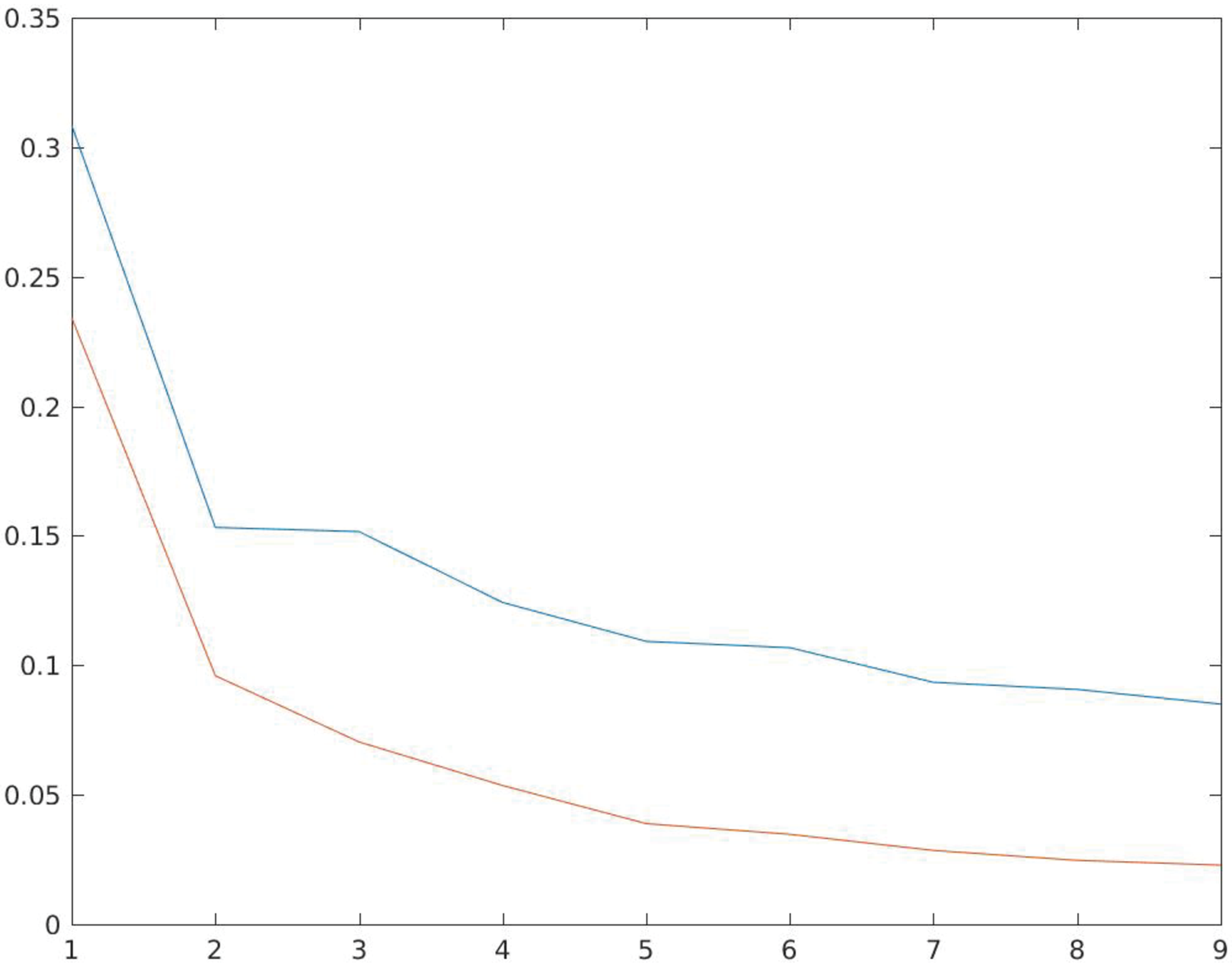} 
\caption{ Example \ref{ex-gauss-3}: Relative errors $\epsilon^{Stokes}_d$ (\textbf{red}) and
$\epsilon_d$ (\textbf{blue})\label{fig:ex3}} 
\end{figure}
}\end{ex}

\begin{ex}
\label{ex-gauss-4}
{\rm We next consider an example in dimension $n=3$. Let $\om_1 = \{\textbf{x} \in \mathbb{R}^3 : (\textbf{x} - \textbf{u})^T \textbf{A}_1 (\textbf{x} - \textbf{u}) \le 1\}$ and $\om_2 = \{\textbf{x} \in \mathbb{R}^2 : (\textbf{x} - \textbf{v})^T \textbf{A}_2 (\textbf{x} - \textbf{v}) \le 1\}$ for the values 

$
\textbf{u} = (0,0,0),
$
$\textbf{v} = (-2,0,-1)$, \[\textbf{A}_1 = \left[\begin{array}{ccc}
-\frac{1}{16} & 0 & 0\\
0 & 1 & 0\\
0 & 0 & \frac{1}{4}
\end{array}
\right]
\mbox{ and } \textbf{A}_2 = \left[\begin{array}{ccc}
\frac{1}{4} & \frac{1}{2} & 0\\
\frac{1}{2} & -1 & 0\\
\frac{1}{4} & \frac{1}{4} & \frac{1}{2}
\end{array}
\right]\]

Results for $d=6$ are displayed in Table \ref{tbl:ex4} and the relative errors $\epsilon_d$ and $\epsilon^{Stokes}_d$  are displayed in  Figure \ref{fig:ex4}.
The quality of results is comparable to that in Examples \ref{ex-gauss-2} and \ref{ex-gauss-3} for $d=6$.
\begin{table}[h] 
\bgroup
\def\arraystretch{1.5}
\begin{tabular}{|c|c|c|c|c|c|}
\hline 
$\overline{\rho}_{6}$ & $\underline{\rho}_{6}$ & $\epsilon_{6}$ & $\overline{\rho}^{Stokes}_{6}$ & $\underline{\rho}^{Stokes}_{6}$ & $\epsilon^{Stokes}_{6}$\\
\hline
$2.8222$ & $2.3123$ & $18\%$ & $2.6856$ & $2.5360$ & $5.6\%$\\ 
\hline
\end{tabular}
\egroup
\caption{Example \ref{ex-gauss-4}: Bounds and relative gap for $d=6$\label{tbl:ex4}}
\end{table}

\begin{figure}[ht]
\includegraphics[width=0.65\columnwidth]{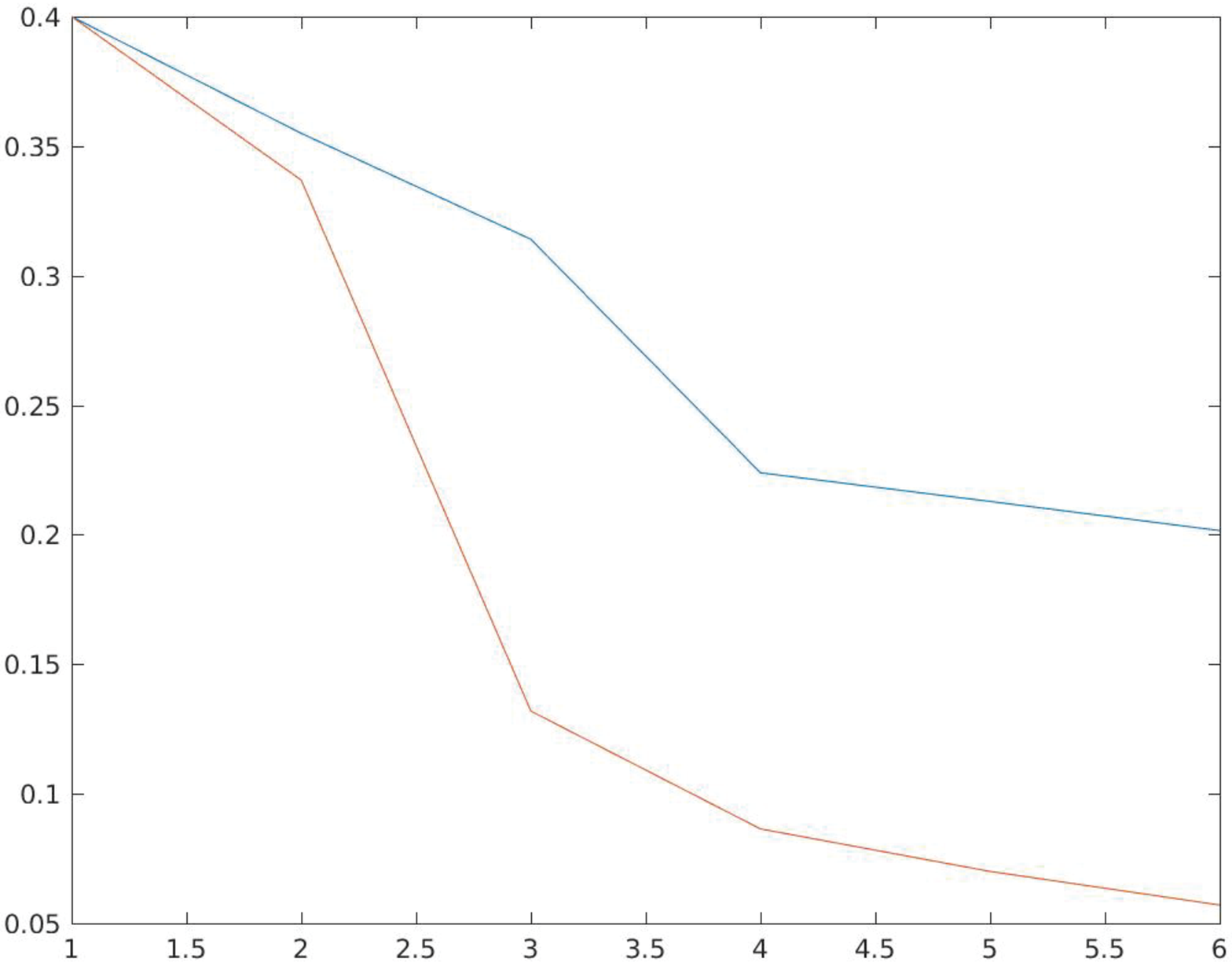} 
\caption{Example \ref{ex-gauss-4}:  Relative error  $\epsilon^{Stokes}_d$ ({\bf red}) and $\epsilon_d$ ({\bf blue}) \label{fig:ex4}}
\end{figure}
}
\end{ex}

\begin{ex}
\label{ex-gauss-5}
{\rm 
Still in dimension $n=3$, let  $\om = \om_1 \cup \om_2$ with $\om_1 = \{\textbf{x} \in \mathbb{R}^3 : \textbf{x}^Te \le 1\}$ and $\om_2 = \{\textbf{x} \in \mathbb{R}^3 : \textbf{x} ^T \textbf{A} \textbf{x}\ \le 1\}$, where 
$e = (1,1,1)$ and \[\textbf{A} = \left[\begin{array}{ccc}
\frac{1}{4} & \frac{1}{2} & 0\\
\frac{1}{2} & -1 & 0\\
\frac{1}{4} & \frac{1}{4} & \frac{1}{2}
\end{array}
\right].\]
The relative errors $\epsilon_d$ and $\epsilon^{Stokes}_d$  are displayed in  Figure \ref{fig:ex5}.
\begin{table}[h] 
\bgroup
\def\arraystretch{1.5}
\begin{tabular}{|c|c|c|c|c|c|}
\hline 
$\overline{\rho}_{7}$ & $\underline{\rho}_{7}$ & $\epsilon_{7}$ & $\overline{\rho}^{Stokes}_{7}$ & $\underline{\rho}^{Stokes}_{7}$ & $\epsilon^{Stokes}_{7}$\\
\hline
$2.8143$ & $2.3494$ & $17\%$ & $2.6887$ & $2.5338$ & $6\%$\\ 
\hline
\end{tabular}
\egroup
\caption{Example \ref{ex-gauss-5}: Bounds and relative gaps for $d=7$
\label{tbl:ex5}}
\end{table}

\begin{figure}[ht]
\includegraphics[width=0.65\columnwidth]{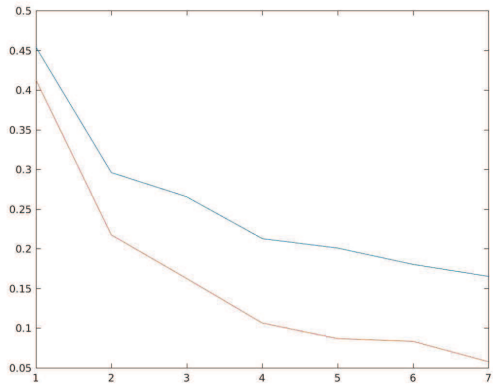} 
\caption{Example \ref{ex-gauss-5}:  Relative error $\epsilon^{Stokes}_d$ ({\bf red}) and $\epsilon_d$ ({\bf blue}) \label{fig:ex5}}
\end{figure}
}\end{ex}

One can see that in all examples quite good approximations are obtained with relatively few moments (up to order $2d\leq 18$ for $n=2$ and 
$2d\leq 14$ for $n=3$) provided that
we use the hierarchy (\ref{primal-newsdp-stokes}) with the additional  moments constraints induced by Stokes' formula. The convergence 
of the hierarchy (\ref{primal-newsdp}) (without those Stokes constraints) is indeed much slower. 

For all the examples that we have treated,
the (crucial) moment and localizing matrices involved in (\ref{primal-newsdp}) and in (\ref{primal-newsdp-stokes}) 
have been expressed
in the canonical basis $(\x^\alpha)_{\alpha\in\N^n}$ of monomials for 
simplicity and easyness of implementation of the SDP relaxations. But this choice is in fact
the worst from a numerical point of view (numerical stability and robustness) which prevented us from
solving (\ref{primal-newsdp}) and (\ref{primal-newsdp-stokes}) for $d\geq 7$ when $n=3$. It is very likely that
the basis of orthonormal polynomials w.r.t. $\mu$ (Legendre for the Lebesgue measure $\mu$ on $[-1,1]$ and
Hermite for the Gaussian measure $\mu$) is a much better (and recommended) choice. Such a more sophisticated implementation was beyond the scope of this paper.

\section*{Conclusion}

In this paper we have provided a numerical scheme to approximate as closely as desired the measure 
$\mu(\om)$ of a finite {\it union} $\om=\cup_{i=1}^p\om_i$ of basic semi-algebraic sets (the case of a single basic semi-algebraic set was treated in \cite{lass-gaussian})).
Surprisingly, even though the case of a union of semi-algebraic sets complicates matters significantly we are still able to adapt the methodology developed in \cite{lass-gaussian}
and provide a monotone non-increasing (resp. non-decreasing) sequence of upper (resp. lower)  bounds that converges to $\mu(\om)$ as the number of moments considered increases. In addition we are also able to use additional moment constraints induced by an appropriate application of Stokes' Theorem 
which permits to improve significantly the convergence. In fact those additional moment constraints are crucial to obtain good bounds rapidly as they 
permit strongly attenuate a Gibbs' phenomenon that otherwise appears.

Our current implementation could be significantly 
improved by using a basis for polynomials more appropriate than the usual canonical basis of monomials 
(the worst choice from a numerical stability point of view). For instance in doing so it should
be possible to implement step $d=8,9$ of the hierarchy in dimension $n=3$, and step $d=7$ for $n=4$. As the convergence
seems to be fast, each additional step of the hierarchy can yield a significant improvement.

The methodology was presented for the Lebesgue measure
$\mu$ when $\om$ is compact and the Gaussian measure for non-compact sets $\om$, but in fact and remarkably, the same methodology works for any measure 
$\mu$ that satisfies Carleman's condition and provided that all its moments are available (or can be computed easily). 

Of course the methodology proposed in this paper  is computationally expensive, especially when compared with Monte-Carlo type methods. 
But the latter provide only an estimate of $\mu(\om)$
and by no means an upper or lower bound on $\mu(\om)$ and therefore these two types of methods should be seen as complementary rather than 
competing.
In its present form it is also limited to small dimension problems (typically $n\leq 3,4$)
because since each upper (or lower) bound requires to solve a semidefinite program whose size increases fast in the hierarchy, one is 
limited by the current efficiency of state-of-the-art semidefinite solvers. However to the best of our knowledge this is the first method that 
provides a sequence of  upper and lower bounds with strong asymptotic guarantees, at least at this level of generality.

\subsection*{Acknowledgement}
Research  funded by by the European Research Council
(ERC) under the European Union's Horizon 2020 research and innovation program
(grant agreement ERC-ADG 666981 TAMING)"


\begin{thebibliography}{las}
\bibitem{handbook}
Anjos M. , Lasserre J.B. (Eds.). {\it Handbook of Semidefinite, Conic and Polynomial Optimization}, Springer, New York, 2012.
\bibitem{Bollobas}
Bollob\'as B. Volume estimates and rapid mixing. In: {\it Flavors of Geometry}, 
MSRI Publications 31, 1997, pp. 151--180.
\bibitem{cousins1}
Cousins B., Vempala S. A cubic algorithm for computing gaussian volume. Proceedings of the 2014
ACM-SIAM Symposium on Discrete Algorithms (SODA14), Portland, January 2014.
\bibitem{cousins2}
Cousins B., Vempala S. A Practical Volume Algorithm, Math. Program. Comput. {\bf 8}, pp. 133--160, 2016.
\bibitem{curto1}
Curto R.E., Fialkow L.A. {\em Flat extensions of positive moment matrices: recursively generated relations}, 
Memoirs. Amer. Math. Soc. {\bf 136}, AMS, Providence, 1998.
\bibitem{curto2}
Curto R.E.,  Fialkow L.A. The truncated K-moment problem in several variables, {\em J. Operator Theory} {\bf 54}, pp.  189--226, 2005.
\bibitem{dunford}
Dunford N., J. Schwartz. {\em Linear Operators. Part I: General Theory}, John Wiley \& Sons, Inc., New York, 1958.
\bibitem{dyer1}
Dyer M.E., Frieze A.M. The complexity of computing the volume of a polyhedron,
SIAM J. Comput. {\bf 17}, pp. 967--974, 1988.
\bibitem{dyer2}
Dyer M.E., Frieze A., Kannan R. A random polynomial-time algorithm for approximating
the volume of convex bodies, J. ACM {\bf 38}, pp. 1--17, 1991.
\bibitem{sirev}
Henrion D., Lasserre J.B., Savorgnan C.
Approximate volume and integration for basic semialgebraic sets,   SIAM Review {\bf 51},  pp. 722--743, 2009.
\bibitem{gloptipoly}
Henrion, Lasserre J.B., Lofberg J.
Gloptipoly 3: moments, optimization and semidefinite programming, Optim. Methods \& Softwares {\bf 24},   pp. 761--779, 2009.
\bibitem{lass-decomp}
Lasserre J.B.. Lebesgue decomposition in action via semidefinite relaxations, Adv. Comput. Math. {\bf 42},  pp. 1129--1148, 2016.
\bibitem{lass-gaussian}
Lasserre J.B.. Computing gaussian and exponential measures of semi-algebraic sets,
{\tt arXiv:1508.06132}, 2015. submitted.
\bibitem{lass-book-icp}
Lasserre J.B. {\em Moments, Positive Polynomials and Their Applications}, Imperial College Press, London, 2010
\bibitem{newlook}
Lasserre J.B. A new look at nonnegativity on closed sets and polynomial optimization, SIAM J. Optim. {\bf 21},  pp. 864--885, 2011.
\bibitem{vempala1}
Lov\'asz L., Vempala S. Simulated annealing in convex bodies and an 
$O^*(n^4)$  volume algorithm. J. Comput. Syst. Sci., {\bf 72}, pp. 392--417, 2006.
\bibitem{niederreiter}
Niederreiter N. {\it Random Number Generation and Quasi-Monte Carlo Methods}, Society for Industrial and Applied Mathematics,
Philadelphia, 1992.
\bibitem{putinar}
Putinar M. Positive polynomials on compact semi-algebraic sets,
Ind. Univ. Math. J. {\bf 42}, pp. 969--984, 1993.
\bibitem{strong}Trnovsk\'a M. Strong duality conditions in semidefinite programming. {\em J.Elec. Eng.} {\bf 56}, pp. 1--5, 2005.
\end{thebibliography}
\end{document}